\newif\ifcheck
\checktrue

\documentclass[12pt, final]{amsart} 

\usepackage[backend=biber]{biblatex}
\usepackage[stretch=10]{microtype}
\usepackage{hyperref}
\usepackage{cleveref}
\usepackage{extpfeil} 

\usepackage{enumerate}
\usepackage{amssymb,amsfonts,amsmath,amsthm}
\usepackage{tikz}
\usepackage{tikz-cd} 
\usepackage{showkeys} 
\usepackage[breakable]{tcolorbox}
\usepackage{xcolor} 
\usepackage{framed}
\usepackage{todonotes}
\usepackage{cancel}
\usepackage{pifont}
\usepackage{mathrsfs}

\usepackage{tgpagella,eulervm}

\definecolor{shadecolor}{rgb}{0.3,0.7,0.9}

\tikzset{node distance=2cm, auto}

\theoremstyle{remark}
\newtheorem{example}{Example}[section]
\newtheorem{remark}[example]{Remark}

\theoremstyle{definition}
\newtheorem{definition}[example]{Definition}
\theoremstyle{plain}
\newtheorem{proposition}[example]{Proposition}
\newtheorem{corollary}[example]{Corollary}

\newtheorem{theorem}[example]{Theorem}
\newtheorem{lemma}[example]{Lemma}

\usepackage{hyperref}
\hypersetup{colorlinks=true,urlcolor=magenta,citecolor=magenta,linkcolor=magenta}

\newtcolorbox{highlight}{colback=yellow!20!white, colframe=yellow}

\newtcolorbox{checkbox}{colback=red!5!white,
colframe=red!75!black, fonttitle=\bfseries, title={To be removed}, breakable=true}

\DeclareFieldFormat{postnote}{#1}
\DeclareFieldFormat{multipostnote}{#1}

\newcommand{\rsa}{\rightsquigarrow}

\newcommand{\trans}[1]{{\lfloor #1 \rfloor}}

\newcommand{\ZZ}{\mathbb{Z}}
\newcommand{\QQ}{\mathbb{Q}}

\newcommand{\sO}{\mathcal{O}}

\newcommand{\fX}{{\mathfrak X}}
\newcommand{\fY}{{\mathfrak Y}}

\newcommand{\spec}{{\rm Spec}}

\renewcommand{\sl}{{\rm SL}}

\newcommand{\et}{{\acute{e}t}}

\DeclareMathOperator{\sch}{Sch}

\DeclareMathOperator{\bA}{{\bf A}}
\newcommand{\cS}{\cancel{S}}

\DeclareMathOperator{\id}{id}

\DeclareMathOperator{\Spec}{Spec}

\DeclareMathOperator{\Br}{Br}

\DeclareMathOperator{\inn}{inn}

\DeclareMathOperator{\Hom}{Hom}

\newextarrow{\xbigtoto}{{20}{20}{20}{20}}
   {\bigRelbar\bigRelbar{\bigtwoarrowsleft\rightarrow\rightarrow}}

   \tikzset{
    labl/.style={anchor=south, rotate=90, inner sep=.5mm}
}

\makeatletter
\newcommand{\colim@}[2]{%
  \vtop{\m@th\ialign{##\cr
    \hfil$#1\operator@font colim$\hfil\cr
    \noalign{\nointerlineskip\kern1.5\ex@}#2\cr
    \noalign{\nointerlineskip\kern-\ex@}\cr}}%
}
\newcommand{\colim}{%
  \mathop{\mathpalette\colim@{\rightarrowfill@\textstyle}}\nmlimits@
}
\makeatother

\title[The \'{e}tale Brauer-Manin Obstruction]{ The \'{e}tale Brauer-Manin obstruction for classifying stacks}

\keywords{Classifying stacks, quotient stacks, strong approximation, Brauer-Manin obstructions, Galois twists}

\subjclass[2020]{Primary 14G12, 11G35, 14G05, Secondary 14A20}

\author[Ajneet Dhillon]{Ajneet Dhillon}
\address{Ajneet Dhillon, Department of Mathematics, University of Western Ontario, London, Ontario, Canada, N6A 5B7}
\email{adhill3@uwo.ca}

\author[Nicole Lemire]{Nicole Lemire}
\address{Nicole Lemire, Department of Mathematics, University of Western Ontario, London, Ontario, Canada, N6A 5B7}
\email{nlemire@uwo.ca}

\author[Jonathan Martin]{Jonathan Martin}
\address{Jonathan Martin, Department of Mathematics, University of Western Ontario, London, Ontario, Canada, N6A 5B7}
\email{jmart383@uwo.ca}

\author[Yidi Wang]{Yidi Wang}
\address{Yidi Wang, Department of Mathematics, University of Western Ontario, London, Ontario, Canada, N6A 5B7}
\email{ywan6443@uwo.ca}

\begin{document}

\begin{abstract}  We study the strong approximation for classifying stacks~$BG$, where $G$ is a linear algebraic group over a number field $k$. More specifically, we prove that the \'etale Brauer-Manin obstruction is the only obstruction to strong approximation for $BG$. To prove the result, we formulate the theory of torsors and Galois twists for algebraic stacks.
\end{abstract}

\maketitle

\section{Introduction}

The study of \emph{local-global principles} for rational points on algebraic varieties over number fields is an active and long-standing research area in arithmetic geometry. For an algebraic variety $X$ over a number field $k$, one asks whether the existence of a $k_v$-point at all places of $k$ guarantees the existence of a~$k$-point. \emph{Strong approximation} answers a stronger local-to-global type of question via the density of points: whether points given locally at certain places (possibly infinitely many) of $k$ 
can be approximated arbitrarily closely by a $k$-point on ~$X$, equivalently, whether the natural map $X(k) \to X(\bA_k)$ is dense.

One intrinsic explanation of the failure of local-global principles is the 
\emph{Brauer-Manin obstruction}: there is a pairing $\left< \, ,\, \right>\colon X(\bA_k) \times \Br(X) \to \QQ/\ZZ$, and the Brauer-Manin locus $X(\bA_k)^{\Br}$ contains all adelic points that vanish under the pairing for all classes in $\Br(X)$. Comparing to $X(\bA_k)$, $X(\bA_k)^{\Br}$ is a finer set that contains $X(k)$.
For a broad class of algebraic varieties, it was shown that the image of $X(k) \to X(\bA_k)^{\Br}$ is dense. Namely, the Brauer-Manin obstruction is the \emph{only} obstruction to strong approximation. See for example, \cite{borovoi-demarche:13}, \cite{Demarche2010}, \cite{Harari2008}, etc.
This framework naturally extends to algebraic stacks, the~$k$-points of which are families of algebraic objects encoded with symmetries. 

In this manuscript, we address this question for classifying stacks $BG$ for linear algebraic groups $G$, the ~$k$-points of which classify ~$G$-torsors. The study of strong approximation for $BG$ therefore translates to that of density problems of~$G$-torsors.

A recent work by the first author showed that the Brauer-Manin obstruction is the only obstruction to strong approximation on $BG$ when $G$ is a connected linear algebraic group (\cite[Theorem 5.5]{Dhillon:25}). That is, the image of $BG(k) \to BG(\bA_k)^{\Br}$ is dense. When $G$ is disconnected, it is unclear whether the analogous result holds, especially when $G/G^{\circ}$ is not abelian. We instead consider a finer obstruction, the \emph{\'etale Brauer-Manin obstruction}. For an algebraic stack $\fX$, we define the \emph{\'etale Brauer-Manin locus} as 
\begin{equation*}
    \fX(\bA_k)^{\et, \Br}\coloneqq \bigcap_{\substack{\fY \overset{f}{\to} \fX \\ H \textnormal{-torsor}}} \coprod_{\sigma \in H^1(k,H)} f^{\sigma}(\fY^{\sigma}(\bA_{k})^{\Br}), 
\end{equation*} where $H$ is taken over all finite \'etale group schemes over $k$. It is a finer obstruction set and captures more subtle arithmetic phenomena that might be invisible at the original level.
In particular, we have proven the following result.

\begin{theorem}\label{thm:intro} 
Let $G$ be a linear algebraic group over a number field $k$. Let $S$ be a finite collection of places of $k$ containing all infinite places of $k$.
    \begin{enumerate}
        \item The \'etale Brauer-Manin locus $BG(\bA_{\cS,k})^{\et,\Br}$ is closed.
        \item The image of $BG(k)\to BG(\bA_{\cS,k})^{\et,\Br}$ is dense. That is, the \'etale Brauer-Manin obstruction is the only obstruction to strong approximation for $BG$ off $S$.
    \end{enumerate}
\end{theorem}

We work under the topology developed in \cite{Dhillon:25}. The main ingredient in the proof is the formalization of Galois twists for stacks. Though the theory is well-known for varieties and schemes, developing the analogous theory for stacks requires considerable work in 2-Categories. We adopt the setup about group actions on stacks in \cite{ginot} and study properties of torsors over an algebraic stack and their Galois twists.

Another significance of Theorem \ref{thm:intro} is its potential application to Malle's Conjecture. The inverse Galois problem asks whether any finite group $G$ is the Galois group of a Galois extension $L/k$, and Malle's Conjecture predicts the number of such Galois extensions with bounded discriminant, fixing a group $G$. Note that $G$-Galois extensions over~$k$ are connected $G$-torsors. In \cite{Ellenberg--Satriano--Zureick-Brown2023}, Ellenberg--Satriano--Zurieck-Brown defined height functions on algebraic stacks and translated Malle's Conjecture to counting~$k$-points on $BG$ with bounded height. 
A recent work by Loughran--Santens (see \cite{loughran-santens25}) showed that there is a deep correlation between Malle's conjecture and cohomological obstructions on $BG$. The authors would like to explore this direction in the future.

The paper is organized as follows. In Section \ref{sec:torsors}, we review the theory of torsors over algebraic stacks and define Galois twists of algebraic stacks. In Section \ref{sec:ebm}, we define the étale Brauer-Manin obstruction for stacks and prove that the corresponding locus is closed. In Section \ref{sec:maintheorem}, we present the proof of our main theorem. Finally, in Appendix \ref{appendix}, we present preliminaries and technical details for group actions on algebraic stacks.

\section*{Notation}

\begin{tabular}{cl}
    $k$ & Our ground field. It is assumed to be a number field.\\
    $\Omega_k$ & The set of places of $k$. \\
    $k_v$ & The completion of $k$ with respect to the valuation $v$.\\
    $\sO_v$ & The valuation subring of $k_v$.\\
    $\sO_{k,S}$ & The ring of $S$-integers in $k$. That is, $\sO_{k,S}\coloneqq \left\{x\in k \mid v(x)\geq 0, \forall v \not \in S\right\}$.\\
    $\bA_{k,S}$ & For some finite subset $S \subseteq \Omega_k$ this is the ring $\prod_{v\in S}k_v \times \prod_{v \not\in S}\sO_v$.\\
    $\bA_k$ & The ring of adeles of $k$.\\
    $\bA_{\cS,k}$ & For some finite subset $S\subseteq \Omega_k$ this is the image of of $\bA_k$ under the \\ & projection onto the product $\prod_{v\not\in S} k_v$. \\
    $G$ & A linear algebraic group over $k$. \\ 
    $G^\circ$ & The connected component of the identity of $G$. \\ 
    $BG$ & The classifying stack of $G$-torsors.
\end{tabular}


\section{Torsors over algebraic stacks}\label{sec:torsors}
The reader is referred to \cite{romagny:05} for a discussion of group actions on stacks. We assume that the reader is familiar with the general theory of torsors over schemes or algebraic spaces. 
We will consider algebraic stacks in the fppf-topology on schemes. 

The following lemma is elementary. 

\begin{lemma}\label{l:isomorphism-criteria-for-stacks}
    Let $f:\fX\to\fY $ be a morphism of algebraic stacks. 
    \begin{enumerate}
        \item Suppose that there exists a presentation $Y\to \fY$ such that $\fX\times_{\fY} Y $ is an algebraic space. Then~$f$ is representable by algebraic spaces.
        \item Suppose that there exists an fppf surjective morphism $Y\to \fY$ such $\fX\times_{\fY} Y \to Y$ is an isomorphism. Then $f$ is an isomorphism.
    \end{enumerate}
\end{lemma}

\begin{proof}
    Since the presentation morphism $Y\to \fY$ is smooth and surjective, the first assertion follows from 
    \cite[\href{https://stacks.math.columbia.edu/tag/04ZP}{04ZP}]{stacks-project} since smooth morphisms are fppf.
    
    For the second assertion, \cite[\href{https://stacks.math.columbia.edu/tag/04ZP}{04ZP}]{stacks-project} and the 
    assumption that $\fX\times_{\fY} Y\cong Y$ is an algebraic space, show that $f$ is representable by algebraic spaces.
    Then fppf descent shows that $f:\fX\to \fY$ is an isomorphism.
    See \cite[\href{https://stacks.math.columbia.edu/tag/04XD}{04XD}]{stacks-project} and property (15) in \cite[\href{https://stacks.math.columbia.edu/tag/04XB}{Section 04XB}]{stacks-project}.
\end{proof}

By a group acting on an algebraic stack we will mean a weak group action as in \cite{ginot} and \cite{romagny:05}. See \Cref{appendix} for details.

\begin{definition}\label{d:stack-torsor}
Let $\fY$ be an algebraic stack over $(\sch/k)_{\mathrm{fppf}}
$ and $G$ a linear algebraic group over $k$. By a $G$-torsor over $\fY$
we mean an algebraic stack $\fX$ with a right action of ~$G$ and a morphism $f:\fX\to \fY$ such that 
\begin{enumerate}
    \item the morphism $f$ is equivariant for the trivial action of $G$ on $\fY $, 
    \item there is an fppf-cover $Y\to \fY$ by an algebraic space such that after base change to $Y$ the morphism $Y\times_{\fY}\fX  \to Y$ is the trivial $G$-torsor over $Y$. 
\end{enumerate}    
\end{definition}

\begin{remark}\label{r:basic}
    Note that in the above definition, the trivial $G$-torsor $Y \times_\fY \fX$ an algebraic space, so by \Cref{l:isomorphism-criteria-for-stacks} (1)
    $f$ is representable. Moreover, one sees that the canonical morphism
    $ \fX \times G \to \fX \times_\fY \fX$ is an isomorphism by using 
    \Cref{l:isomorphism-criteria-for-stacks} (2)
    and noting that after base change to $Y$ one obtains an isomorphism:
    $$(Y \times_\fY \fX) \times G \to  (Y \times_\fY \fX) \times_Y (Y \times_\fY \fX) \cong Y \times_\fY (\fX \times_\fY \fX)$$
\end{remark}

Proposition (\ref{p:basicExample}) below produces some interesting examples of torsors. 

\subsection{Galois Twists}\label{subsec:Galois-twist}

Let $f: \fX \to \fY$ be a $G$-torsor and let $F$ be a scheme with a left action of $G$. Then the product $\fX\times F$ has a right $G$-action 
where $G$ acts via the inverse on the $F$ component. We define $\fX\times^{G}F $ to be the quotient stack 
$$
\fX\times^{G}F := [\fX\times F/G],
$$
and denote the canonical map $\fX \times^G F \to \fY$ by $f^F$. 
See \cite[4.2]{ginot}.

\begin{remark}
    Note that Galois twists commutes with base change, i.e.,  
    $(Y\times_\fY \fX) \times^G F \cong Y \times_\fY (\fX \times^G F)$.
    In light of this, one can write 
    $Y \times_\fY \fX \times^G F$ without ambiguity.
\end{remark}

In the case that $F$ is a left torsor corresponding to a Galois cohomology class $\sigma\in H^1(k, G)$ for a field $k$, we denote 
$\fX \times^{G} F:= \fX^\sigma $ and the associated morphism as 
$$
f^{\sigma}: \fX^{\sigma}\to \fY. 
$$

Recall that if $P$ is a left $G$-torsor over $k$ then the corresponding inner form of $G$ is 
defined as the quotient of $G\times P$ by the action 
$$
(g,p)\longmapsto (g^h,h^{-1}p) \quad \text{ where } g^{h}=h^{-1}gh
$$
We denote this inner form by $G\times^{G,\inn} P$. We will also sometimes write 
$G^P \coloneqq G \times^{G,\inn}P$ when context is clear. See \Cref{e:twist} for more details.

\begin{proposition}
    Let $\fX \to \fY$ be a $G$-torsor. Suppose that $P$ is a left $G$-torsor over $k$. Then 
    $$
    \fX\times^{G}P \to \fY 
    $$
    is a $G^P$-torsor. 
\end{proposition}
\begin{proof}
   
    The required action is constructed in \ref{p:innerAction}.
It is then routine to check that the induced map 
    $\fX \times^G P \to \fY$ is $G^P$-equivariant when ~$\fY$ is given the
    trivial $G^P$ action. This shows (1) in \Cref{d:stack-torsor}.

    To check (2), we take $Y \to \fY$ to be 
    an fppf-cover such that $Y \times_\fY \fX \to Y$ is the trivial 
    ~$G$-torsor over $Y$. Then we have 
    $Y \times_\fY \fX \times^G P \cong Y \times_\fY G \times^G P \cong Y \times_\fY G^P \to Y$
    is the trivial ~$G^P$-torsor over $Y$. 
\end{proof}

\subsection{The basic example}
Consider a short exact sequence of linear algebraic groups
$$
1\to N \to G\to G/N\to 1 
$$ where $N$ is a normal subgroup of $G$.
For a right $G$-torsor $P$, $P\times^{G}G/N$ has a right action of $G/N$.

Let $Red(G\rsa N)$ denote the category fibered in groupoids over $\left(\sch/k\right)$ defined by 
\begin{enumerate}
    \item the objects over a scheme $S$ are pairs $(P,s)$ where $P$ is a right $G$-torsor over $S$ and~$s$ is a section of $P\times^{G}G/N$; 
    \item the morphisms over $S$ are isomorphisms of torsors compatible with sections. 
\end{enumerate}

It is well known that $Red(G\rsa N)$ is an algebraic stack.
\begin{lemma}
    There is a strict right action of $G/N$ on $Red(G\rsa N)$. 
\end{lemma}

\begin{proof}
    Let $H \coloneqq G/N$. Given an object $(P,s)$ of 
    $Red(G\rsa N)$ over a scheme $S$, note that $P \times^G H$ is an $H$-torsor. 
    An element $h \in H(S)$ induces an automorphism of $P \times^G H$.
    We will abuse notation and denote this automorphism by $h$. Then $h\circ s$
    is a section of $P \times^G H$. The action of $H$ on $Red(G\rsa N)$ on 
    objects is given by $(P,s)\cdot h := (P,h\circ s)$.

    We then show that this group action is strict. 
    Let $e \in H(S)$ denote the unit element. The induced automorphism on $P \times^G H$ is the identity. Thus, 
    $(P,s)\cdot e = (P,s)$. Additionally, note that 
   $$
   (P,s)\cdot h_1\cdot h_2 = (P,h_2\circ h_1 \circ s) = (P,h_1h_2\circ s) = (P,s)\cdot h_1h_2
   $$  This shows that the action on objects is strict.

  
Recall that a morphism $\varphi \colon (P,s) \to (Q, s')$ is a morphism $\phi \colon P\to Q$ that is compatible with the sections, i.e., $\phi \circ s = s'$. An action of $h$ on $\varphi$ yields a morphism $\varphi.h \colon (P, h\circ s) \to (Q, h\circ s')$, by noting that $\phi$ is $H$-equivariant, i.e., $h \circ \phi = \phi \circ h$. It follows immediately that this action is strict as well.  
\end{proof}

We will denote by $BG$ the classifying stack of \textit{right} torsors over $G$. 
The following lemma is well-known. We provide an algebraic proof for the convenience of the readers.

\begin{lemma}\label{l:reduction}
    There is an equivalence of categories fibered in groupoids
    $$ BN \cong Red(G\rsa N) $$
    Hence, there is a weak action of $G/N$ on $BN$. 
\end{lemma}

\begin{proof}
    \Cref{e:actions} shows that equivalences of categories preserve weak actions,
    but not necessarily strict actions. As such, the second assertion follows from the 
    first assertion and the previous lemma. 


Define 
        \begin{align*}
            \Phi \colon BN(S) &\to Red(G\rsa N)(S) \\
            Q &\mapsto P \coloneqq Q \times^{N} G
        \end{align*} For $\Phi$ to be well-defined, we have to check $P \times^G G/N$ admits a section. Note that
        \[
            P \times^G G/N  \cong Q \times^{N} G \times^G G\times^{G/N} \Spec(k) \cong Q \times^{N} G/N
        \] Since $N$ fixes the identity coset $G/N$, it admits an $N$-equivariant map $\Spec(k) \to G/N$. Applying $Q\times^{N} \cdot$ leads to a section $\Spec(k) \to Q \times^{N} G/N \cong  P \times^G G/N$. It is left to the reader to verify the functoriality of this construction.

    

    
    Next, define $\psi: Red(G \rsa N) \to BN$, 
    where $\psi(P,s)$ is defined by the cartesian square:
    $$
    \begin{tikzcd}
        \psi(P,s) & P \\
        S & {P\times^G G/N}
        \arrow[from=1-1, to=1-2]
        \arrow[from=1-1, to=2-1]
        \arrow["\lrcorner"{anchor=center, pos=0.125}, draw=none, from=1-1, to=2-2]
        \arrow[from=1-2, to=2-2]
        \arrow["s", from=2-1, to=2-2]
    \end{tikzcd}
    $$
    Here $P \to P \times^G G/N \cong P/N$ is  is an $N$-torsor. Then it follows that $\psi(P,s) \to S$ is an ~$N$-torsor. Once again, functoriality is left to the reader.

    Next, we show that $\psi\circ \phi \cong \id_{BN}$.
    Let $Q$ be an $N$-torsor over $S$. By construction, to show that 
    $\psi\circ\phi(Q) \xrightarrow{\sim} Q$,  we check that the following diagram is 
    cartesian:
    $$
    \begin{tikzcd}
        Q & {Q\times^NG} \\
        S & {(Q\times^N G)\times^G G/N}
        \arrow[from=1-1, to=1-2]
        \arrow[from=1-1, to=2-1]
        \arrow[from=1-2, to=2-2]
        \arrow["{\sigma_Q}", from=2-1, to=2-2]
    \end{tikzcd}
    $$ by noting that 
    $$
    (Q\times^N G)\times^G G/N \cong Q \times^N G/N \cong Q/N \times G/N \cong S \times Q/N.
    $$

    It now remains to show that
    $\phi\circ\psi \cong \id_{Red(G\rsa N)}$. Let $(P,s) \in Red(G\rsa N)(S)$.
    Since $s:S \to P \times^G G/N$ is $N$-equivariant, the map 
    $\psi(P,s) \to P$ from the cartesian diagram defining $\psi(P,s)$ is also 
    $N$-equivariant. This induces a $G$-equivariant morphism of $G$-torsors
    $\phi(\psi(P,s)) = \psi(P,s) \times^N G \to P$, which is an isomorphism. It follows by 
    construction that this isomorphism is compatible with the sections 
    of $\psi(P,S)\times^NG$ and~$P$. 
    This completes the proof.
\end{proof}

\begin{proposition}\label{p:basicExample}
    The canonical map $BN\to BG$ is a $G/N$-torsor.
\end{proposition}

\begin{proof}By \Cref{l:reduction}, there is a strict action of $G/N$ on $BN$, so one just needs to 
    verify the local triviality condition in \Cref{d:stack-torsor}. Using the description of the action in terms of 
    $Red(G\rsa N)$ one see that the morphism $BN\to BG$ is equivariant. This follows from 
    the 2-cartesian diagram:
    $$
    \begin{tikzcd}
        {G/N} & {\text{Spec}(k)} \\
        BN & BG
        \arrow[from=1-1, to=1-2]
        \arrow[from=1-1, to=2-1]
        \arrow[from=1-2, to=2-2]
        \arrow[from=2-1, to=2-2]
    \end{tikzcd}
    $$
\end{proof}

In what follows we consider a possibly disconnected linear algebraic group $G$ over ~$k$. We denote by $G^{\circ}$ the connected component of the 
identity. Note that $G^{\circ}$ is defined over ~$k$ and is a normal subgroup. We write $H:= G/G^{\circ}$. 

Let $P$ be a left $H$-torsor representing some Galois cohomology class  $\sigma\in H^1(k,H)$. We can form the twisted stack 
$$(BG^{\circ})^{\sigma} \coloneqq BG^{\circ}\times^{H}P = [BG^{\circ}\times P/H]. 
$$
Consider a $k$-scheme $T$. 
By \cite[Theorem 4.1]{romagny:05}, a $T$-point of this twist consists of a pair $(Q,\Phi)$,
where $Q\to T$ is a $H$-torsor and $\Phi:Q\to BG^{\circ}\times P$ is an equivariant map. 



\begin{proposition}\label{p:inner} Let notation be as above. 
     The twisted stack $(BG^{\circ})^{\sigma}$ is isomorphic to the quotient stack $[P/G]$.
\end{proposition}

\begin{proof}
    Let $T$ be a $k$-scheme.   
    A~$T$-point of $[P/G]$ consists of a pair $(R,\varphi)$ where $R\to T$ is a $G$-torsor and $\varphi\colon R\to P$ is a~$G$-equivariant map. Since $G^{\circ}$ acts trivially on $P$, we have $P/G^{\circ} \cong P$, and we see that this descends to a diagram
    $$
    \begin{tikzcd}
        R & P \\
        {R/G^\circ} & {P/G^\circ}
        \arrow["\varphi", from=1-1, to=1-2]
        \arrow[from=1-1, to=2-1]
        \arrow["\cong", from=1-2, to=2-2]
        \arrow["{\overline{\varphi}}", from=2-1, to=2-2]
    \end{tikzcd}
    $$ Note that $R/G^{\circ}$ is an $H$-torsor, and $\overline{\varphi}$ is $H$-equivariant. Moreover, $R \to R/G^\circ$ is a~$G^\circ$-torsor, and therefore we have a $R/G^\circ$-point of $BG^\circ$. The morphism $R/G^{\circ} \to BG^{\circ}$ is obviously $H$-equivariant. Combining them, we obtain an $H$-equivariant map $R/G^{\circ} \to BG^{\circ} \times P$, which yields a well-defined map
    \[
       \Psi \colon \left[P/G\right](T) \to \left[BG^{\circ}\times P/H\right](T) \colon R \mapsto R/G^{\circ}.
    \] It remains to show that $\Psi$ is an isomorphism. By part (2) of Lemma \ref{l:isomorphism-criteria-for-stacks}, it suffices to show that it is an isomorphism after base change to $\bar{k}$. 
Over $\bar{k}$, the $H$-torsor $P$ is trivial, so it suffices to show that $\left[H/G\right] \to BG^{\circ}$ is an isomorphism. 
    This is immediate.
\end{proof}

\section{The \'etale Brauer-Manin Obstruction}\label{sec:ebm}

\subsection{A review of topologizing the adelic points of a stack}

Let $\fX/k$ be a finite type algebraic stack. In this subsection we will recall how the topology on the adelic points of $\fX$ is defined in \cite[Section 2]{Dhillon:25}. 
Note that we can ``spread out'' $\fX$ over $\sO_{k,S}$, for some finite set $S\subseteq \Omega_k$. Namely, there exists an $\sO_{k,S}$-stack $\mathcal{X}$ with generic fibre $\fX$. See \cite[Theorem 4.2]{Dhillon:25}.

We further assume that $S$ is a finite subset containing all infinite places.\footnote{This is assumed in \cite[Definition 2.14]{Dhillon:25}} 
The adelic points of $\fX$ are defined as 
$$
    \fX(\bA_{k,S}) \coloneqq \prod_{v\in S} \fX(k_v) \times \prod_{v\not\in S} \mathcal{X}(\sO_{v}) 
$$ and note that this expression does not depend on the choice of $\mathcal{X}$.

\begin{definition}
    \label{d:lifting-presentation} Let notation be as above. We say $\fX$ is \emph{$S$-liftable} if the diagonal morphism $\mathcal{X} \to \mathcal{X} \times_{\sO_{k,S}} \mathcal{X}$ is separated, and there is a presentation $P\to \fX$ such that after spreading 
    out to $\mathcal{P} \to \mathcal{X}$ over $\sO_{k,S}$, we have
    \begin{enumerate}
        \item $\mathcal{P}$ is a finite type separated $\sO_{k,S}$-algebraic space,
        \item for all $T\supseteq S$ and every $s\in \fX(\bA_{k,T})$, there is a lift of $s$ to $P(\bA_{k,T})$,
        \item the induced map $\mathcal{P}(\sO_{v})\to \mathcal{X}(\sO_{v})$ is surjective for all but finitely many $v\not\in S$. 
    \end{enumerate}
    In this situation, we will call $P\to \fX$ an \emph{$S$-lifting presentation} of $\fX$. 
\end{definition}\noindent Such presentations exist for a large class of stacks. See the discussion right after \cite[Definition 2.14]{Dhillon:25}.    

Given an $S$-lifting presentation $P \to \fX$, we give $\fX(\bA_{k,S})$ the quotient topology via $f:P(\bA_{k,S})\to \fX(\bA_{k,S})$. 
One can show that this does not depend on choice of lifting presentation (see \cite[Proposition 2.17]{Dhillon:25}). 
The topology on $\fX(\bA_k)$ is then defined by the colimit of topological spaces:
$$
\fX(\bA_k) = \colim_{T\subseteq \Omega_k} \fX(\bA_{k,T}). 
$$

\subsection{The \'etale Brauer-Manin locus}

Let $f:\fX\to \fY$ be a $H$-torsor, where $H$ is a linear algebraic group over $k$. The following result from descent theory is well-known.

\begin{lemma}\label{lemma:descent}
    Let $S$ be a $k$-scheme. Let $\sigma \in H^1(S,H)$ and let $f^{\sigma}\colon \fX^{\sigma} \to \fY$ be the twisted torsor constructed in Section~\ref{subsec:Galois-twist}. Then
    \[
        \fY(S) = \coprod_{\sigma \in H^1(S,H)}f^{\sigma}(\fX^{\sigma}(S)).
    \]
\end{lemma}

\begin{proof}
    See \cite[Theorem 8.4.1]{poonen:17}. The proof remains the same by replacing varieties by stacks.
\end{proof}

Let $f$ be as above. Suppose further that $H$ is a finite \'etale group scheme. We let 
\begin{equation}\label{eq:f-locus}
    \fY(\bA_k)^{f,Br} = \coprod_{\sigma \in H^1(k,H)} f^\sigma(\fX^{\sigma}(\bA_k)^{Br}). 
\end{equation}

If $\fY$ is a finite type algebraic stack, we define the \emph{\'etale Brauer-Manin locus} by 
$$
\fY(\bA_k)^{\et,Br} =\bigcap_{H}  \bigcap_{\substack{\fX\overset{f}{\to} \fY \\ H\textnormal{-torsor}}}\fY(\bA_k)^{f,Br},
$$ where $H$ is taken over all finite \'etale group schemes over $k$. 
In the rest of this section, we show the \'etale Brauer-Manin locus is closed in $\fY(\bA_k)$.

\begin{proposition}\label{p:proper}
    Let $f:\fX\to \fY$ be a proper representable morphism of stacks. If $\fY$ has an~$S$-lifting presentation by a separated algebraic space, then 
    the induced morphism on adelic points is topologically proper. In particular, $f:\fX(\bA_{k}) \to \fY(\bA_{k})$ is closed.
\end{proposition}

\begin{proof}
        Consider an $S$-lifting presentation $p:Y\to \fY$ with $Y$ a separated algebraic space. Then $q:Y\times_{\fY}\fX\to \fX$ is an $S$-lifting presentation 
        and $Y\times_{\fY}\fX$ is a separated algebraic in view of the properness of $f$ by \cite[Lemma 2.16]{Dhillon:25}. Consider finite $T\subseteq \Omega_{k}$ with $S\subseteq T$ and the 
        Cartesian diagram:
        \begin{center}
            \begin{tikzcd}
                Y\times_{\fY}\fX  \ar[r, "g"] \ar[d,"q"] & Y \ar[d,"p"] \\
                \fX \ar[r, "f"]& \fY. 
            \end{tikzcd}
        \end{center}
        By \cite[Proposition 5.8]{conrad2020}, the morphism $g$ is topologically closed on $\bA_{k,T}$-points. 
        Let $Z\subseteq \fX(\bA_{k,T})$ be a closed subset. Then $q^{-1}(Z)$ is closed. Then we have 
        $g(q^{-1}(Z)) = p^{-1}(f(Z))$ and it follows that $f(Z)$ is closed. 
        The result follows from the construction of the colimit topology. 
\end{proof}

\begin{remark}\label{r:torsor}
    Suppose that $f:\fX\to\fY$ is a $H$-torsor and $\fY$ has an $S$-lifting 
    presentation. Then the morphism $f$ satisfies the hypothesis of the proposition, so does each twist 
    $f^{\sigma}$ for $\sigma\in H^1(k,H)$. 

    The assertion for $f$ follows from \Cref{r:basic}. For the Galois twists, one applies \Cref{p:inner}. 
\end{remark}

    \begin{proposition} \label{p:finite} Suppose $H$ is a finite linear algebraic group over $k$.
    Let $T \subseteq \Omega_k$ be a finite subset containing all infinite places. 
    Then the set 
    $$
    \mathcal{T} \coloneqq \{ \sigma\in H^1(k,H) \mid f^\sigma(\fX^\sigma (\bA_k))\cap \fY(\bA_{k,T})\not= \emptyset\}
    $$
    is finite. 
\end{proposition}
\begin{proof} 
The argument is essentially from \cite[Proposition 4.4]{harari02}. Indeed, we can ``spread out''~$H$ to 
an $\sO_{k,S}$-group scheme~$\mathcal{H}$, $\fX$ to an $\sO_{k,S}$-stack $\mathcal{X}$ and $\fY$ to an 
$\mathcal{H}$-torsor $\mathcal{Y}$, where $S \subseteq \Omega_k$ is a finite set containing $T$.
For any $\sigma$ in the set 
$$
\mathcal{S} \coloneqq \{ \sigma\in H^1(k,H) \mid f^\sigma(\fX^\sigma (\bA_k))\cap \fY(\bA_{k,S})\not= \emptyset\},
$$ and any $v \not\in S$, the pullback of $f^{\sigma}$ along $\Spec(k_v)\to \fY$ is an $H_{k_v}$-torsor 
that corresponds to $\sigma_v \in H^1(k_v, H)$. Moreover, $\sigma_{v}$ must be in the image of the 
natural map $H^1(\sO_v, \mathcal{H}) \to H^1(k_v, H)$. This implies that 
$\sigma$ must be unramified at all $v \not \in S$. Since $\sigma$ corresponds to an \'etale $k$-algebra, 
which is essentially a product of finite separable extensions of $k$, this implies that these fields 
extensions are unramified over $v\not\in S$. By \cite[V.4 Theorem 5]{lang}, the collection of such 
unramified field extensions is finite, and therefore the set $\mathcal{S}$ is finite as well. 

Moreover, since for $T\subseteq S$, the natural map $\fY(\bA_{k,T})\subseteq \fY(\bA_{k, S})$ is an inclusion, we must have that $\mathcal{T}$ is finite as well.
\end{proof}

\begin{proposition}\label{p:f-Br-closed} 
    The set $\fY(\bA_k)^{f, Br}$ defined in (\ref{eq:f-locus}) is a closed subset of $\fY(\bA_k)$. 
\end{proposition}

\begin{proof}
    By  \Cref{p:proper}, each of the sets $f^\sigma(\fX^{\sigma}(\bA_k)^{Br})$ are closed. 
    Since $\fY(\bA_k) = \colim\fY(\bA_{k,T})$, to show $\fY(\bA_k)^{f,Br}$ is closed, it suffices to show that
    for every finite $S\subseteq \Omega_k$ containing all infinite places,  $\fY(\bA_{k,T}) \cap \fY(\bA_k)^{f,Br}$ is closed 
    in $\fY(\bA_{k,T})$. To see this, we observe:
    $$
    \fY(\bA_{k,T}) \cap \fY(\bA_k)^{f,Br} = \coprod_{\sigma \in H^1(k,H)} (\fY(\bA_{k,T}) \cap f^\sigma(\fX^{\sigma}(\bA_k)^{Br})).
    $$ 
    By \Cref{p:finite}, this is a finite union of closed sets and is thus closed.
\end{proof}

\begin{corollary}\label{cor:et-BM-locus-closed}
    The \'etale Brauer-Manin locus is closed. 
\end{corollary}

\begin{proof}
    This follows directly from Proposition \ref{p:f-Br-closed}.
\end{proof}

\begin{remark}\label{r:cancel_S}
    In \Cref{p:finite}, we may replace $\bA_k$ and $\bA_{k, T}$ by $\bA_{\cancel{S}, k}$ and $\bA_{\cancel{S}, k, U}$, respectively, for a finite set $S\subseteq \Omega_k$ containing all infinite places and a finite set $U \subseteq \Omega_k$ containing~$S$ and prove an analogous result. Therefore, one may show that the set $\fY(\bA_{\cancel{S},k})^{f, Br}$ is a closed subset of $\fY(\bA_{\cancel{S},k})$ as in \Cref{p:f-Br-closed}. 
\end{remark}

\subsection{The \'etale Brauer-Manin obstruction}

\begin{lemma}\label{l:closed}
    Let $\fY$ be an algebraic stack of finite type. Then the \'etale Brauer-Manin locus $\fY(\bA_k)^{\et, \Br}$ contains the closure of the diagonal image of $\fY(k) \to \fY(\bA_k)^{\et, \Br}$. That is, the \'etale Brauer-Manin locus captures all $k$-rational points of $\fY$.
\end{lemma}
\begin{proof}
    The fact that the diagonal image of $\fY(k) \to \fY(\bA_k)$ is contained in $\fY(\bA_k)^{\et, \Br}$ follows immediately from Lemma \ref{lemma:descent} and the fact that the Brauer-Manin locus $\fX(\bA_k)^{\Br}$ contains the diagonal image of $\fX(k)\to \fX(\bA_k)$ for any algebraic stack $\fX$. 
    By Corollary~\ref{cor:et-BM-locus-closed},~$\fY(\bA_k)^{\et, \Br}$  is closed, so it contains the closure of $\fY(k)$.
\end{proof}

    \begin{remark}
        By \Cref{r:cancel_S}, we may replace $\bA_k$ by $\bA_{\cancel{S},k}$ for a finite set $S\subseteq \Omega_k$ containing all infinite places and prove an analogous result in \Cref{l:closed}.
    \end{remark}

Now we are ready to define the \'etale Brauer-Manin obstruction for strong approximation.

\begin{definition}\label{def:ebm-obstruction-off-S}
    Let $\fY$ be an algebraic stack of finite type. We say 
    \emph{strong approximation for $\fY$ off a finite set of places $S$ holds with respect to the \'etale Brauer-Manin obstruction} 
    if the image of the natural map $\fY(k) \to \fY(\bA_{\cS,k})^{\et, \Br}$ is dense.    
\end{definition}

\section{The main result}\label{sec:maintheorem}

Recall our standing notation 
$$
1\to G^{\circ}\to G \to H \to 1
$$
where $G$ is a linear algebraic group and $G^{\circ}$
is the connected component of the identity. 

\begin{lemma}
    \label{l:homogeneous}
    Let $P$ be a $H$-torsor over $k$. We think of $P$ as a homogeneous space over $G$. Let $G\hookrightarrow \sl_n$ be a faithful representation. 
    Then $P\times^{G}\sl_{n}$ is a homogeneous space for $\sl_n$ with connected stabilizer. 
\end{lemma}

\begin{proof}
    We can verify the statement by passing to the algebraic closure $\bar{k}$. Hence, we may assume that $P$ is the trivial $H$-torsor. Then $P\times^{G}\sl_{n} \cong \sl_{n}/G^{\circ}$ and the result follows. 
\end{proof}

\begin{theorem}\label{theorem:strong-app-twist}
     Let $\sigma \in H^1(k, H)$. Strong approximation holds for the twist $(BG^{\circ})^{\sigma}$ off $S$ with respect to $\Br((BG^{\circ})^{\sigma})$. 
\end{theorem}

\begin{proof} Note that we can always find a faithful representation $G \to \sl_n$.  Then by Proposition \ref{p:inner} and \cite[Lemma 2.1]{Dhillon:25}, 
    $(BG^\circ)^\sigma = [P \times^G \sl_n/\sl_n]$. Therefore, $P \times^G\sl_n \to (BG^\circ)^\sigma$ is an $\sl_n$-torsor.
    By \cite[Corollary 2.3]{Dhillon:25}, any $\sl_n$-torsor is trivial over $\spec(\bA_{\cS,k})$, so any adelic point $x \in (BG^\circ)^\sigma(\bA_{\cS,k})$ must lift to $\tilde{x} \in [P \times^G\sl_{n}/\sl_n](\bA_{\cS,k})$. Let $q$ denote the lifting map. 

    By \cite[Proposition 5.3]{Dhillon:25}, $\Br((BG^\circ)^{\sigma}) \twoheadrightarrow \Br(P\times^G \sl_n)$. Therefore, if $x \in (BG^\circ)^{\sigma}(\bA_{\cS,k})^{\Br}$, then $\tilde{x} \in  P\times^G\sl_n(\bA_{\cS,k})^{\Br}$. This is easily seen from the Brauer-Manin paring $\left<x, t\right> = \left<\tilde{x}, q^{\ast}(t)\right> = 0$.

    Let $U$ be an open neighborhood of $x$ in the adelic topology. Then $q^{-1}(U)$ is a open neighborhood of $\tilde{x}$ in the adelic toplogy. By Lemma \ref{l:homogeneous} and \cite[Theorem 0.1]{borovoi-demarche:13}, strong approximation holds for $P\times^G\sl_n$, so there exists an equivariant rational point $y \in q^{-1}(U)$. Hence, $q(y) \in U$ is a rational point. This proves the claim.
\end{proof}

\begin{theorem}
    Strong approximation for $BG$ off a finite set of places $S$ holds with respect to the \'etale Brauer-Manin obstruction.
\end{theorem}

\begin{proof}
    We must show that the image of $BG(k) \to BG(\bA_{\cS, k})^{\et,\Br}$ is dense.
    By Lemma~\ref{lemma:descent} and Proposition~\ref{p:basicExample}, 
    \[
        BG(k) = \coprod_{\sigma \in H^1(k,H)} (BG^{\circ})^\sigma(k)
    \] whence any $k$-point of $BG$ lifts to a $k$-point of $(BG^{\circ})^\sigma$ for some $\sigma \in H^1(k, H)$. 

    By Theorem~\ref{theorem:strong-app-twist}, the image of $(BG^{\circ})^\sigma(k) \to (BG^{\circ})^{\sigma}(\bA_{\cS,k})^{\Br}$ is dense. Therefore, the image of 
    \[
        BG(k) = \coprod_{\sigma \in H^1(k,H)} (BG^{\circ})^\sigma(k) \to \coprod_{\sigma\in H^1(k,H)} (BG^{\circ})^{\sigma}(\bA_{\cS,k})^{\Br},
    \] is dense. Moreover, by definition
    \[
        BG(\bA_{\cS, k})^{\et,\Br} \subseteq \coprod_{\sigma\in H^1(k,H)} (BG^{\circ})^{\sigma}(\bA_{\cS,k})^{\Br},
    \] so the image of $BG(k) \to BG(\bA_{\cS, k})^{\et,\Br}$ is dense.
\end{proof}

\begin{example}
    In the case where $G$ is a finite group, we have $G^\circ$ is trivial,
    and $H = G/G^\circ = G$. It follows that for 
    $\sigma \in H^1(k,H) = H^1(k,G)$, we have
    $(BG^\circ)^\sigma = \Spec k$. Thus, the Brauer-Manin obstruction
    of $(BG^\circ)^\sigma$ is trivial, and the closure of $BG(k)$ 
    in $BG(\bA_k)$ is:
    $$\coprod_{\sigma \in H^1(k,G)}(BG^\circ)^\sigma(\bA_{\cS,k})^{\Br} = H^1(k,G)= BG(k)$$
    That is, the $k$-points of $BG$ are closed in the adelic topology in this case.
\end{example}


\appendix
\section{Right group actions on stacks and groupoids}\label{appendix}

In this appendix, we will restate some definitions and results of \cite{ginot} in terms of right actions. Then we will state the necessarily details for group actions when construction Galois twists of an algebraic stack.

\begin{definition}\label{d:weakAction}    
    Let $\fX$ be a category fibered in groupoids. 
    A right action of $G$ on $\fX$ is a triple 
    $(\mu, \alpha, a)$ where $\mu: \fX\times G\to \fX$ is a morphism such that 
    the following diagrams are 2-commutive:
    $$
    \begin{tikzcd}
        {\fX\times G \times G} & {\fX\times G} && {\fX\times G} & \fX \\
        {\fX\times G} & \fX && \fX
        \arrow["{\id_\fX\times m}", from=1-1, to=1-2]
        \arrow["{\mu\times \id_G}"', from=1-1, to=2-1]
        \arrow["\mu", from=1-2, to=2-2]
        \arrow["\mu", from=1-4, to=1-5]
        \arrow["\alpha", Rightarrow, from=2-1, to=1-2]
        \arrow["\mu"', from=2-1, to=2-2]
        \arrow["{\id_\fX\times \{e\}}", from=2-4, to=1-4]
        \arrow[""{name=0, anchor=center, inner sep=0}, "{\id_\fX}"', from=2-4, to=1-5]
        \arrow["a"', Rightarrow, from=0, to=1-4]
    \end{tikzcd}
    $$
    where $e$ is the identity of $G$. 
    In what follows we will write $\mu(x,g)=x\cdot g$ where $x$ is an object (or a morphism) 
    of  $\fX$. Given $(x,g,h)\in \fX\times G\times G$ we write 
    $\alpha^{x}_{g,h}:x\cdot g \cdot h \to x \cdot gh$ and $a^x:x \to x\cdot e$
    for the value of the natural transformations $\alpha$ and $a$ respectively on 
    these objects. The above data is subject to the axioms:
    \begin{enumerate}
        \item $\alpha_{gh,k}^{x}\circ (\alpha_{g,h}^{x}\cdot k) = \alpha_{g,hk}^{x} \circ \alpha_{h,k}^{x\cdot g}$ as morphisms $x\cdot g \cdot h \cdot k \to x\cdot ghk$.
        \item $1_{x\cdot g}=(\alpha^{x}_{e,g})\circ (a^{x}\cdot g)$ as morphisms $x\cdot g \to x \cdot g$.
        \item $1_{x\cdot g}=(\alpha^{x}_{g,e})\circ (a^{x\cdot g})$ as morphisms $x\cdot g \to x \cdot g$. 
    \end{enumerate}
    A \emph{strict action of $G$} on $\fX$ is a weak action where $\alpha$ and $a$ are both the identity. 
\end{definition}

If $G$ is a group acting weakly on $\fX$ we denote the 
\emph{transformation groupoid} (\cite[\S 3]{ginot}) by $\lfloor\fX/G \rfloor$.  Recall that it is the category fibered in groupoids with 
$$
{\rm obj}\lfloor\fX/G \rfloor = {\rm obj}\fX
$$
and 
\begin{align*}
    \Hom_{\lfloor\fX/G \rfloor}(x,y) &= \{ (\gamma, g) \ | \  g\in G,\  \gamma\in \Hom_\fX(x,y\cdot g)\} \\ 
    &= \coprod_{g\in G} \Hom_{\fX}(x,y\cdot g). 
\end{align*} 
Given $(\gamma, g): x \to y$ and $(\delta, h): y \to z$, the composition is defined by: 
$$
(\delta , h) \circ (\gamma, g) = (\alpha_{h,g}^{z}\circ(\delta\cdot g)\circ \gamma, hg) 
$$

\begin{example}
    \label{e:actions}
    One way of generating actions is from an equivalence of categories fibered in 
    groupoids. Say $F:\fX\to \fY$ is such an equivalence, with quasi-inverse 
    given by $G:\fY \to \fX$. If $\fY$ has a strict action of a group then 
    $\fX$ acquires a weak action of the same group. The action on $\fX$ is given 
    by $x\cdot g := G(F(x)\cdot g)$. One checks that this is a weak action 
    by noting that $G\circ F \sim \id_\fX$, and so:
    $$x\cdot 1 = G(F(x)\cdot 1) = G\circ F(x) \cong x $$
    $$x\cdot g \cdot h = G(F\circ G(F(x)\cdot g)\cdot h) \cong G(F(x)\cdot g \cdot h) = G(F(x)\cdot gh) = x\cdot gh$$
    In this case, the action on $\fX$ will be strict if and only if the equivalence
    is in fact an isomorphism of categories. Also note that the action on $\fY$ need 
    not be strict in order to give $\fX$ a weak action in this way.
\end{example}

\begin{definition}
    \label{d:equivariant}
    Suppose that $\fX$ and $\fY$ have weak $G$-actions given by $(\mu_\fX,\alpha, a)$ 
    and $(\mu_\fY,\beta, b)$ respectively. An \emph{equivariant morphism $\fX\to \fY$} is given by a 
    pair $(F,\sigma)$ where $F:\fX \to \fY$ is a morphism of 
    stacks or groupoids (see \cite[1.3 (ii)]{romagny:05} or \cite[3.2]{ginot}) 
    and $\sigma$ is a 2-morphism such that the following diagram 2-commutes:
    $$
    \begin{tikzcd}
        {\fX\times G} & \fX \\
        {\fY\times G} & \fY
        \arrow["{\mu_\fX}", from=1-1, to=1-2]
        \arrow["{F\times\id_G}"', from=1-1, to=2-1]
        \arrow["F", from=1-2, to=2-2]
        \arrow["\sigma", Rightarrow, from=2-1, to=1-2]
        \arrow["{\mu_\fY}"', from=2-1, to=2-2]
    \end{tikzcd}
    $$
    In this setting, when given $(x,g) \in \fX \times G$, we write 
    $\sigma_g^x:F(x)\cdot g \to F(x\cdot g)$ for the value of the natural 
    transformation $\sigma$ on this object. With this notation, we also require:
    \begin{enumerate}
        \item $F(\alpha^{x}_{g,h}) \circ \sigma^{x\cdot g}_{h}\circ (\sigma_{g}^{x}\cdot h) = \sigma^{x}_{gh}\circ\beta^{F(x)}_{g,h} $ 
        as morphisms $F(x)\cdot g \cdot h \to F(x\cdot gh)$
        \item $(\sigma_{e}^{x})^{-1}\circ F(a^{x}) = b^{F(x)}$ as morphisms $F(x) \to F(x)\cdot e$ 
    \end{enumerate}
\end{definition}

Given a $G$-equivariant morphism $(F,\sigma):\fX\to\fY$, there is an induced morphism 
$\lfloor F\rfloor :\lfloor \fX/G\rfloor \to \lfloor \fY/G\rfloor $ defined on morphisms
$(\gamma,g):x \to y$ by: 
$$
\lfloor F \rfloor (\gamma,g) =  ((\sigma_{g}^{y})^{-1}\circ F(\gamma), g) \in \Hom_{\lfloor \fY/G\rfloor}(F(x),F(y))
$$
see \cite[page 14]{ginot}.

\begin{definition}\label{d:equivariant2}
    Suppose $\fX$ and $\fY$ have weak $G$-actions given by $(\mu_\fX,\alpha, a)$ 
    and $(\mu_\fY,\beta, b)$ respectively. Given two $G$-equivariant morphisms 
    $(F,\sigma)$ and $(H,\tau): \fX \to \fY$ a \emph{$G$-equivariant 2-morphism }
    $\Lambda : (F,\sigma)\Rightarrow (H,\tau)$ is a 2-morphism such that 
    $$
    \tau^{x}_{g}\circ(\Lambda^{x}\cdot g) = \Lambda^{x\cdot g}\circ \sigma^{x}_{g} \quad \text{ as morphisms } F(x)\cdot g \to H(x\cdot g)  
    $$
\end{definition}

\begin{remark}
    Suppose that $\phi:G_1 \to G_2$ is a group homomorphism, and $\fX$ and $\fY$
    have weak $G_1$ and $G_2$ actions respectively. Then $\phi$ induces 
    a weak $G_1$ action on $\fY$. In this case, it is reasonable to talk about 
    an $\phi$-equivariant morphism $(F,\sigma):\fX \to \fY$. Such a morphism is subject to the 
    same conditions as in \Cref{d:equivariant}, but with $\id_G: G \to G$ in the 
    diagram replaced with $\phi:G_1 \to G_2$. Moreover, given two such 
    $\phi$-equivariant morphisms $(F,\sigma)$ and $(H,\tau)$, it is 
    reasonable to ask for a $\phi$-equivariant 2-morphism $\Lambda:(F,\sigma) \Rightarrow (H,\tau)$.
    In this context we say that $(F,\sigma)$, $(H,\tau)$, and $\Lambda$ are 
    $\phi$-equivariant.
\end{remark}

\begin{proposition}
    \label{p:natExists}
    In the setting of the previous definition, $\Lambda$ induces a 2-morphism:
    $$
    \lfloor \Lambda\rfloor : \lfloor F\rfloor \Rightarrow \lfloor H\rfloor
    $$
    defined by:
    $$
    \lfloor \Lambda \rfloor^{x} = (b^{H(x)}\circ\Lambda^{x},e). 
    $$  
\end{proposition}

\begin{proof}
    This amounts to checking that for every $(\gamma,g):x \to y$ in $\lfloor \fX\rfloor$
    the following diagram commutes in $\lfloor \fY/G\rfloor$: 
    $$
    \begin{tikzcd}
        {\lfloor F\rfloor(x)} & {\lfloor H\rfloor(x)} \\
        {\lfloor F\rfloor(y)} & {\lfloor H\rfloor(y)}
        \arrow["{\lfloor \Lambda\rfloor^{x}}", from=1-1, to=1-2]
        \arrow["{\lfloor F\rfloor(\gamma,g)}"', from=1-1, to=2-1]
        \arrow["{\lfloor H\rfloor(\gamma,g)}", from=1-2, to=2-2]
        \arrow["{\lfloor \Lambda\rfloor^{y}}"', from=2-1, to=2-2]
    \end{tikzcd}
    $$
    If one unpacks the definitions, they see that this is asking for the 
    outermost square of the following diagram to commute in $\fY$:
    $$
    \begin{tikzcd}[column sep=huge]
        {F(x)} & {H(x)} & {H(x)\cdot e} \\
        {F(y\cdot g)} & {H(y\cdot g)} & {H(y)\cdot g \cdot e} \\
        {F(y)\cdot g} & {H(y)\cdot e \cdot g} & {H(y)\cdot g}
        \arrow["{\Lambda^x}", from=1-1, to=1-2]
        \arrow["{F(\gamma)}"', from=1-1, to=2-1]
        \arrow["{b^{H(x)}}", from=1-2, to=1-3]
        \arrow["{H(\gamma)}"', from=1-2, to=2-2]
        \arrow["{((\tau_{g}^{y})^{-1}\circ H(\gamma))\cdot e}", from=1-3, to=2-3]
        \arrow["{\Lambda^{y\cdot g}}"', from=2-1, to=2-2]
        \arrow["{(\sigma_{g}^y)^{-1}}"', from=2-1, to=3-1]
        \arrow["{(\tau_{g}^{y})^{-1}}"{description}, from=2-2, to=3-3]
        \arrow["{\beta_{g,e}^{H(y)}}", from=2-3, to=3-3]
        \arrow["{(b^{H(y)}\circ \Lambda^y)\cdot g}"', from=3-1, to=3-2]
        \arrow["{\beta_{e,g}^{H(y)}}"', from=3-2, to=3-3]
    \end{tikzcd}
    $$
    The top left square commutes by the naturality of $\Lambda$. 
    To show that the bottom left corner commutes, one recalls by (2) in \Cref{d:weakAction} 
    that $1_{H(y)\cdot g} = (\beta_{e,g}^{H(y)})\circ (b^{H(y)}\cdot g)$, so we have:
    \begin{align*}
        \beta_{e,g}^{H(y)}\circ(b^{H(y)}\circ\Lambda^y)\cdot g \circ (\sigma_{g}^{y})^{-1}
        &= \beta_{e,g}^{H(y)}\circ (b^{H(y)}\cdot g)\circ(\Lambda^y\cdot g) \circ (\sigma_{g}^{y})^{-1}\\
        &= 1_{H(y)\cdot g} \circ (\Lambda^y\cdot g) \circ (\sigma_{g}^{y})^{-1}\\
        &= (\tau_{g}^{y})^{-1}\circ \Lambda^{y\cdot g}
    \end{align*}  where the last equality follows from \Cref{d:equivariant2}.  
    
    The commutativity of the 
    top right corner of the diagram depends only on $(H,\tau)$ and has nothing to do 
    with $(F,\sigma)$ or $\Lambda$. The commutativity can be shown using the naturality
    of $\tau$, (2) in \Cref{d:equivariant} and (3) in \Cref{d:weakAction}.    
\end{proof}

\begin{example}\label{e:inner} This example is well-known. We present it here in terms of the sheafification of transformation groupoids. 
    Let $G$ be a linear algebraic group over a field $k$ and $P$ a left $G$-torsor over ~$k$. 
    Then there is a right action of $G$ on $G \times P$ given by:
    $$
    (g,p) \cdot h := (g^h,h^{-1}p) \quad \text{ where } g^{h}=h^{-1}gh 
    $$
    The space obtained by taking the quotient of this action is called the \emph{inner form}
    associated to $P$, and is denoted by $G^{P}$. It is in fact a group. To construct the 
    multiplication map consider the morphism 
    $$
    G\times P \times G\times P \to G\times P
    $$
    which, for a $k$-scheme $S$, is given on $S$-points by:
    $$
    (g_{1},p,g_{2}, p') \mapsto (g_{1} g_{3}^{-1}g_{2}g_3,p)
    $$ where $g_3$ is the unique $S$-point of $G$ such that $p' = g_3\cdot p$. 

    There is an action of $G \times G$ on $G\times P\times G\times P$ and an action of 
    $G$ on $G\times P$. One then checks that this morphism is $\pi_{2}$-equivariant, where 
    $\pi_2:G\times G\to G$ is the projection onto the second factor. 

    This gives a morphism of transformation groupoids 
    $$
    \trans{G\times P\times G \times P/G^2}\to \trans{G\times P/G}
    $$
    and hence a morphism $G^{P}\times G^{P}\to G^{P}$ after sheafifying. 
    Similarly, to construct inversion and unit maps, one considers the morphisms 
    $\iota:G \times P \to G \times P$ given by $(g,p) \mapsto (g^{-1},p)$, and 
    $\eta:P \to G \times P$ given by $p \mapsto (e,p)$ respectively. Since these
    morphisms are $G$-equivariant when $P$ is given a right $G$-action by 
    inverting its left $G$-action, it follows that these morphisms also give 
    morphisms of transformation groupoids. After sheafifying, one gets 
    morphisms $G^P \to G^P$ and $* \to G^P$. One then readily checks that 
    $G^{P}$ becomes a group under these morphisms.
\end{example}

\begin{example}\label{e:twist}
    Let $G$ be a linear algebraic group over a field $k$, let $m:G \times G \to G$
    denote multiplication. Let $\fX$  be a stack with a weak right action of $G$ given by 
    $(\mu_1,\alpha, a)$ and let $P$ be a left $G$-torsor with action denoted 
    by $\mu_2$. There is a morphism $\mu:\fX \times P \times G \to \fX \times P$
    given by:
    $$ (x,p,g) \mapsto (x\cdot g,g^{-1}p) $$
    that induces a weak right action of $G$ on $\fX \times P$. The weak action is 
    determined by the natural isomorphisms 
    $\zeta:\mu \circ (\mu \times \id_G) \Rightarrow \mu \circ (\id_{\fX\times P} \times m)$
    given by:
    $$\zeta_{g,h}^{x,p} = \alpha_{g,h}^x \times \id_{h^{-1}g^{-1}p}: (x\cdot g\cdot h, h^{-1}g^{-1}p) \to (x\cdot gh, h^{-1}g^{-1}p)$$
    and $z:\id_{\fX\times P} \Rightarrow \mu\circ(\id_{\fX}\times e)$ 
    given by:
    $$z^{x,p} = a^x \times \id_{p}: (x,p) \to (x \cdot e, p) $$
    We will denote the quotient stack obtained by this action by $\fX^{P}$. 
\end{example}

\begin{proposition}\label{p:innerAction}
    In the situation of the above examples,
    the morphism 
    $$
    \rho :  \fX\times P \times  G\times P\to \fX\times P
    $$
    given by 
    $$
    (x,p,g,p')\mapsto (x\cdot (h^{-1}gh),p), 
    $$ where $h$ is the unique point of $G$ that satisfies $h\cdot p = p'$,
    induces a weak group action 
    $$
    \bar{\rho}: \fX^{P}\times G^{P}\to \fX^{P}.
    $$
\end{proposition}

\begin{proof}
    There is a $G\times G$ action on $\fX\times P \times  G\times P$ given by combining the actions of the previous two examples. Explicitly
    $$
    (x,p, g, hp)\cdot (s_{1},s_{2}) = (x\cdot s_{1}, s_{1}^{-1}p,g^{s_{2}}, s_{2}^{-1}hp). 
    $$
    We need to show that $\rho$ is $\pi_1$-equivariant where $\pi_1:G\times G\to G$ is the 
    projection onto the first factor. That is , if we let $(\mu_{1},\alpha,a)$ 
    (resp. $\mu_{2}$) be the action of $G$ on $\fX$ (resp. $P$) then we need to give a 
    2-morphism $\sigma$ such that the following diagram 2-commutes:
    $$
    \begin{tikzcd}
        {\fX\times P\times G \times P \times G \times G} & {\fX\times P \times G \times P} \\
        {\fX\times P \times G} & {\fX\times P}
        \arrow[from=1-1, to=1-2]
        \arrow["{\rho\times\pi_1}"', from=1-1, to=2-1]
        \arrow["\rho", from=1-2, to=2-2]
        \arrow["\sigma", Rightarrow, from=2-1, to=1-2]
        \arrow["\mu"', from=2-1, to=2-2]
    \end{tikzcd}
    $$
    Given $(x,p,g,hp,s_1,s_2) \in \fX\times P \times G \times P \times G \times G$,
    then by going along the bottom path of the above diagram, we get 
    $(x\cdot g^{h} \cdot s_{1},s_{1}^{-1}p)$. Going along the top we get
    $(x\cdot s_{1}\cdot g^{hs_1}, s_{1}^{-1}p)$. Since the $P$ component of
    both of these pairs are the same, to give an isomorphism between them it suffices 
    to give an isomorphism 
    $$x\cdot g^{h}\cdot s_{1} \to x\cdot s_{1}\cdot g^{hs_1}$$
    in $\fX$. One then checks that if 
    $\omega_{s}^{x,g} = (\alpha_{s,g^{s}}^{x})^{-1} \circ \alpha_{g,s}^{x}$
    then $\omega_{s_{1}}^{x,g^{h}}$ is the desired isomorphism. Thus 
    $$\sigma_{(s_1,s_2)}^{(x,p,g,hp)} = \omega_{s_{1}}^{x,g^{h}} \times \id_{s_{1}^{-1}p}$$
    
    This proves the existence of the morphism $\bar{\rho}$. It remains to check that 
    $\bar{\rho}$ can be extended to an action $(\bar{\rho}, \bar{\lambda}, \bar{l})$. 
    To do this, we will construct the appropriate natural isomorphisms on 
    $\fX \times P \times G\times P$ and check that they are equivariant. Let $m$
    denote the morphism $G \times P \times G \times P \to G \times P$
    constructed in \Cref{e:inner} which descends to multiplication 
    $G^{P} \times G^{P} \to G^{P}$.

    To construct $\lambda$ is to give a 2-morphism such that the following 
    diagram 2-commutes:
    $$
    \begin{tikzcd}
        {\fX\times P\times G \times P \times G \times P} & {\fX\times P \times G \times P} \\
        {\fX\times P \times G \times P} & {\fX\times P}
        \arrow["{\id_{\fX\times P}\times m}", from=1-1, to=1-2]
        \arrow["{\rho\times\id_{G\times P}}"', from=1-1, to=2-1]
        \arrow["\rho", from=1-2, to=2-2]
        \arrow["\lambda", Rightarrow, from=2-1, to=1-2]
        \arrow["\rho"', from=2-1, to=2-2]
    \end{tikzcd}
    $$
    Given 
    $(x,p,g_{1},h_{1}p,g_{2},h_{2}p) \in \fX\times P\times G \times P \times G \times P$,
    then going along the top path of the above diagram we get
    $(x\cdot g_{1}^{h_{1}}g_{2}^{h_{2}},p)$. Going along the bottom 
    path we get $(x \cdot g_{1}^{h_{1}}\cdot g_{2}^{h_{2}}, p)$.
    Similar to defining $\sigma$, to give an isomorphism from the latter to the former
    it suffices to give an isomorphism
    $$x \cdot g_{1}^{h_{1}}\cdot g_{2}^{h_{2}} \to x\cdot g_{1}^{h_{1}}g_{2}^{h_{2}} $$
    in $\fX$. By definition, 
    $\alpha_{g_{1}^{h_{1}}, g_{2}^{h_{2}}}^{x}$
    is the desired isomorphism. So 
    $$\lambda^{(x,p,g_1,h_1p,g_2,h_2p)}= \alpha_{g_{1}^{h_{1}},g_{2}^{h_{2}}}^{x} \times \id_{p}$$

    To show that $\lambda$ descends as desired requires us to check that it is 
    $\pi_1$-equivariant, where $\pi_1:G^{3} \to G$ is the first projection,
    and $\fX\times P \times G \times P \times G \times P$ is given an action by 
    $G^{3}$ action by combining the actions of $G$ on $\fX \times P$ and 
    $G \times P$. If one unpacks all of the definitions, constructions and computations, 
    this reduces to checking that for all 
    $(x,p,g_{1},h_{1}p,g_{2},h_{2}p,s) \in \fX\times P \times G \times P \times G \times P \times G$
    the following diagram commutes in $\fX$:
    $$
    \begin{tikzcd}[column sep=huge]
        {x\cdot g_{1}^{h_{1}}g_{2}^{h_{2}}\cdot s} & {x\cdot s\cdot g_{1}^{h_{1}s}g_{2}^{h_{2}s}} \\
        {x\cdot g_{1}^{h_{1}}\cdot g_{2}^{h_{2}}\cdot s} & {x\cdot s\cdot g_{1}^{h_{1}s}\cdot g_{2}^{h_{2}s}}
        \arrow["{\omega_{s}^{x,g_{1}^{h_{1}}g_{2}^{h_{2}}}}", from=1-1, to=1-2]
        \arrow["{(\alpha_{g_{1}^{h_{1}},g_{2}^{h_{2}}}^{x})\cdot s}", from=2-1, to=1-1]
        \arrow["{(\omega_{s}^{x,g_{1}^{h_{1}}}\cdot g_{2}^{h_{2}s})\circ\omega_{s}^{x\cdot g_{1}^{h_{1}},g_{2}^{h_{2}}}}"', from=2-1, to=2-2]
        \arrow["{\alpha_{g_{1}^{h_{1}s},g_{2}^{h_{2}s}}^{x\cdot s}}"', from=2-2, to=1-2]
    \end{tikzcd}
    $$
    Recall $\omega_{s}^{x,g} = (\alpha_{s,g^{s}}^{x})^{-1} \circ \alpha_{g,s}^{x}$.
    The commutativity of this diagram can then be shown by repeated uses of (1) in 
    \Cref{d:weakAction}.

    To construct $l$ is to give a 2-morphism such that the following diagram 2-commutes:
    $$
    \begin{tikzcd}
        {\fX\times P \times G \times P} & {\fX \times P} \\
        {\fX \times P}
        \arrow["\rho", from=1-1, to=1-2]
        \arrow["\id_{\fX \times P} \times \eta", from=2-1, to=1-1]
        \arrow[""{name=0, anchor=center, inner sep=0}, "{\id_{\fX\times P}}"', from=2-1, to=1-2]
        \arrow["l"{description}, Rightarrow, from=0, to=1-1]
    \end{tikzcd} 
    $$ Here $\eta$ is the unit map as defined in \Cref{e:inner}.
    Given $(x,p) \in \fX \times P$, then going along the top path we get $(x\cdot e,p)$ and 
    going along the bottom path we get $(x,p)$. So we use:
    $l^{(x,p)} = a^{x} \times \id_p$ to make the diagram 2-commute.
    To show that $l$ descends as desired requires us to check that it is $G$-equivariant.
    Similar to checking that $\lambda$ descends, after unpacking, this problem reduces to checking
    that for all $(x,p,g) \in \fX \times P \times G$ the following diagram commutes in $\fX$:
    $$
    \begin{tikzcd}
        {x\cdot e \cdot g} & {x\cdot g \cdot e} \\
        {x\cdot g} & {x\cdot g}
        \arrow["{\omega_{g}^{x,e}}", from=1-1, to=1-2]
        \arrow["{a^{x}\cdot g}", from=2-1, to=1-1]
        \arrow["{\id_{x\cdot g}}"', from=2-1, to=2-2]
        \arrow["{a^{x\cdot g}}"', from=2-2, to=1-2]
    \end{tikzcd}
    $$
    Since $\omega_{g}^{x,e} = (\alpha_{g,e}^{x})^{-1}\circ\alpha_{e,g}^{x}$
    this follows immediately from (2) and (3) in \Cref{d:weakAction}.

    Thus, both $\lambda$ and $l$ descend as desired. This produces the data required to 
    define a weak group action of $G^P$ on the transformation groupoid $\lfloor \fX/G\rfloor$. 
    One readily checks that it satisfies the axioms of \ref{d:weakAction}. The 
    action on quotient stacks is then obtained by stackifying.
\end{proof}

\newpage

\printbibliography

@article{borovoi-demarche:13,
  author = {Mikhail Borovoi and Cyril Demarche},
  journal = { Comment. Math. Helv. },
  number = {1},
  title = {Manin obstruction to strong approximation for homogeneous spaces. },
  volume = {88},
  year = {2013},
  pages = {1--54}
}

@misc{Dhillon:25,
      title={Approximation theorems for classifying stacks over number fields}, 
      author={Ajneet Dhillon},
      year={2025},
      eprint={2507.13900},
      archivePrefix={arXiv},
      primaryClass={math.NT},
}

@article{Demarche2010,
    author = {Demarche, Cyril},
    title = {Le défaut d’approximation forte dans les groupes linéaires connexes},
    journal = {Proceedings of the London Mathematical Society},
    volume = {102},
    number = {3},
    pages = {563-597},
    year = {2010},
    month = {10},
    issn = {0024-6115},
    doi = {10.1112/plms/pdq033}
}

@article{Ellenberg--Satriano--Zureick-Brown2023, 
  title={Heights on stacks and a generalized Batyrev–Manin–Malle conjecture}, 
  volume={11}, 
  DOI={10.1017/fms.2023.5}, 
  journal={Forum of Mathematics, Sigma}, 
  author={Ellenberg, Jordan S. and Satriano, Matthew and Zureick-Brown, David}, 
  year={2023}, 
  pages={e14}
  }

@article{Harari2008,
  author = {David Harari},
  journal = {Algebra \& Number Theory},
  number = {5},
  title = {Le défaut d'approximation forte pour les groupes algébriques commutatifs},
  volume = {2},
  pages = {595--611},
  year = {2008}
}

@misc{loughran-santens25,
      title={Malle's conjecture and Brauer groups of stacks}, 
      author={Daniel Loughran and Tim Santens},
      year={2025},
      eprint={2412.04196},
      archivePrefix={arXiv},
      primaryClass={math.NT}
}

@book{poonen:17,
    author = {Bjorn Poonen},
    title = {Rational points on varieties},
    publisher = {American Mathematical Society},
    year = {2017},
    series = {{Graduate Studies in Mathematics}},
    volume = {186}
}

@article{romagny:05,
    AUTHOR = {Romagny, Matthieu},
     TITLE = {Group actions on stacks and applications},
   JOURNAL = {Michigan Math. J.},
  FJOURNAL = {Michigan Mathematical Journal},
    VOLUME = {53},
      YEAR = {2005},
    NUMBER = {1},
     PAGES = {209--236}
}

@article {conrad2020,
    AUTHOR = {Conrad, Brian},
     TITLE = {Weil and {G}rothendieck approaches to adelic points},
   JOURNAL = {Enseign. Math. (2)},
  FJOURNAL = {L'Enseignement Math\'ematique. Revue Internationale. 2e
              S\'erie},
    VOLUME = {58},
      YEAR = {2012},
    NUMBER = {1-2},
     PAGES = {61--97}
}

@article {harari02,
    AUTHOR = {Harari, David and Skorobogatov, Alexei N.},
     TITLE = {Non-abelian cohomology and rational points},
   JOURNAL = {Compositio Math.},
  FJOURNAL = {Compositio Mathematica},
    VOLUME = {130},
      YEAR = {2002},
    NUMBER = {3},
     PAGES = {241--273},
}

@book {lang,
    AUTHOR = {Lang, Serge},
     TITLE = {Algebraic number theory},
    SERIES = {Graduate Texts in Mathematics},
    VOLUME = {110},
   EDITION = {Second},
 PUBLISHER = {Springer-Verlag, New York},
      YEAR = {1994},
     PAGES = {xiv+357}
}

@misc{stacks-project,
    author       = {The {Stacks Project Authors}},
    title        = {\textit{Stacks Project}},
    howpublished = {\url{https://stacks.math.columbia.edu}},
    year         = {2025},
  }

@misc{ginot,
      title={Group actions on stacks and applications to equivariant string topology for stacks}, 
      author={Gregory Ginot and Behrang Noohi},
      year={2016},
      eprint={1206.5603},
      archivePrefix={arXiv},
      primaryClass={math.AT}
}

\end{document}